\documentclass[12pt,twoside]{amsart}
\usepackage{amssymb}
\usepackage[all]{xy}

\nonstopmode

\numberwithin{equation}{section} 

\font\tengothic=eufm10 scaled\magstep 1 \font\sevengothic=eufm7
scaled\magstep 1
\newfam\gothicfam
     \textfont\gothicfam=\tengothic
     \scriptfont\gothicfam=\sevengothic

\newtheorem{theorem}{Theorem}[section]
\newtheorem{lemma}[theorem]{Lemma}

\theoremstyle{definition}
\newtheorem{definition}[theorem]{Definition} 
\newtheorem{remark}[theorem]{Remark}
\newtheorem{problem}[theorem]{Problem}
\newtheorem{example}[theorem]{Example}

\newtheorem{notation}[theorem]{Notation}

\newcommand{\Hom}{\operatorname{Hom}}
\newcommand{\Ext}{\operatorname{Ext}}

\newcommand{\cE}{{\mathcal E}}

\newcommand{\cM}{{\mathcal M}}

\newcommand{\cO}{{\mathcal O}}
\newcommand{\cL}{{\mathcal L}}

\newcommand {\ZZ}{\mathbb{Z}}

\newcommand {\PP}{\mathbb{P}}

\begin{document}
\title[Derived category of fibrations]
{Derived category of  fibrations}

\author[L.\ Costa, S. \ Di Rocco, R.M.\ Mir\'o-Roig]{L.\ Costa$^*$, S. \ Di Rocco$^{**}$, R.M.\
Mir\'o-Roig$^{*}$}

\address{Facultat de Matem\`atiques,
Departament d'Algebra i Geometria, Gran Via de les Corts Catalanes
585, 08007 Barcelona, SPAIN } \email{costa@ub.edu}
\address{Department of Mathematics, KTH, SE-10044 Stockholm, Sweden }
\email{dirocco@math.kth.se}
\address{Facultat de Matem\`atiques,
Departament d'Algebra i Geometria, Gran Via de les Corts Catalanes
585, 08007 Barcelona, SPAIN } \email{miro@ub.edu}

\date{\today}
\thanks{$^*$ Partially supported by MTM2010-15256.}
\thanks{$^{**}$ Partially supported by  Vetenskapsr{\aa}det's grant
NT:2006-3539.}



\begin{abstract} In this paper  we give a structure theorem for the derived category $D^b(X)$
 of a Zariski locally trivial fibration $X$ over $Z$ with fiber $F$ provided both $F$ and $Z$ have   a full strongly   exceptional collection of  line bundles.
\end{abstract}



\maketitle



\section{Introduction} \label{intro}
Let $X$ be a smooth projective variety defined over an
algebraically closed field $K$ of characteristic zero and let
$D^b(X)$ be the derived category of bounded complexes of coherent
sheaves of ${\cO}_X$-modules. $D^b(X)$ is one of the most
important algebraic invariants of a smooth projective variety $X$
and, in spite of the increasing interest in understanding its
structure, very little progress has been achieved. The
study of $D^b(X)$ dates back to the late 70's when Beilinson
described the derived category of projective spaces (\cite{Be})
and it has became an important topics in Algebraic
Geometry. Among other reasons this is due to the Homological
Mirror Symmetry Conjecture of Kontsevich \cite{Ko} which states
that there is an equivalence of categories between the derived
category of coherent sheaves on a Calabi-Yau  variety and the
derived Fukai category of its mirror.

\vskip 2mm

An important approach to determine the structure of $D^b(X)$
 is to construct tilting bundles.
Following terminology of representation
theory (cf. \cite{Ba}) a coherent sheaf $T$ of ${\cO}_X$-modules on a
smooth projective variety $X$ is called a {\em tilting sheaf} (or,
when it is locally free, a {\em tilting bundle}) if
\begin{itemize}
\item[(i)] it has no higher self-extensions, i.e.
$\Ext^{i}_X(T,T)=0$ for all $i>0$, \item[(ii)] the endomorphism
algebra of $T$, $A=\Hom_X(T,T)$, has finite global homological
dimension, \item[(iii)] the direct summands of $T$ generate the
bounded derived category $D^b({\cO}_X)$.
\end{itemize}

\vskip 2mm The importance of tilting sheaves relies on the fact
that they can be characterized as those sheaves $T$ of
${\cO}_X$-modules such that the functors
\textbf{R}$\Hom_X(T,-):D^b(X)\longrightarrow D^b(A)$ and
$-\otimes_A^{\mbox{\textbf{L}}}T:D^b(A)\longrightarrow D^b(X)$
define mutually inverse equivalences of the bounded derived
categories of coherent sheaves on $X$ and of finitely generated
right $A$-modules, respectively. The existence of tilting sheaves also
plays  an important role in the problem of characterizing the
smooth projective varieties $X$ determined by their bounded derived category
 of coherent sheaves $D^b(X)$.
 For constructions of tilting
bundles and their relations to derived categories we refer to the
following papers: \cite{Ba}, \cite{Be}, \cite{Bo},
\cite{Ka} and \cite{Or}.

\vskip 2mm In this paper we will focus our attention on the
existence of tilting bundles  which splits as a direct sum of line
bundles on Zariski locally trivial fibrations. The
search for tilting sheaves on a smooth projective variety $X$
splits naturally into two steps. First one has to find the
so-called {\em strongly exceptional collection} of coherent
sheaves on $X$, $(F_0,F_1,\ldots ,F_n)$; then one has to show that $F_0$, $F_1$,
$\ldots  ,$ $F_n$ generate the bounded derived category $D^b(X)$.

\begin{definition} A coherent sheaf $E$ on  a smooth projective variety  $X$ is called {\em exceptional} if
it is simple and $\Ext^{i}_{\cO _X}(E,E)=0$ for $i\ne 0$.  An ordered collection $(E_0,E_1,\ldots ,E_m)$  of coherent
sheaves on $X$ is an {\em exceptional collection} if each sheaf
$E_{i}$ is exceptional and $\Ext^i_{\cO _X}(E_{k},E_{j})=0$ for
$j<k$ and $i \geq 0$.
An exceptional collection $(E_0,E_1,\ldots ,E_m)$ is a {\em
strongly exceptional collection} if in addition $\Ext^{i}_{\cO
_X}(E_j,E_k)=0$ for $i\ge 1$ and  $j \leq k$. If an exceptional collection $(E_0,E_1,\ldots ,E_m)$ of coherent
sheaves on $X$ generates  $D^b(X)$, then it is called {\em full}.
\end{definition}
Thus each
full strongly exceptional collection $(F_0,F_1,\ldots ,F_n)$ defines a tilting sheaf
$T=\oplus_{i=0}^nF_{i}$ and, vice
versa,  each tilting bundle whose direct
summands are line bundles gives rise to a full strongly
exceptional collection.

\vspace{3mm}

The now classical result of Beilinson \cite{Be} states that $(\cO_{\PP^n}, \cO_{\PP^n}(1), \ldots , \cO_{\PP^n}(n))$ is a full strongly exceptional collection on $\PP^n$. The following problem can be considered by now a  natural and important question in Algebraic Geometry.

\begin{problem} Characterize smooth projective varieties which have a full strongly exceptional collection and investigate whether there is one consisting of line bundles.
\end{problem}

Note that not all smooth projective varieties have a full strongly exceptional collection of coherent sheaves. Indeed, the existence of a full strongly
exceptional collection $(E_0,E_1,\ldots,E_m)$ of coherent sheaves
on a smooth projective variety $X$ imposes a rather  strong
restriction on $X$, namely that the Grothendieck group
$K_0(X)=K_0({\cO}_X$-mod$)$ is isomorphic to $\ZZ^{m+1}$.
 Also notice that the above problem contains two subproblems. First to determine the existence of a full strongly exceptional collection on $X$ and then the existence of at least one  consisting of line bundles. We want to stress that this second part is far from being redundant. There are several examples (for instance, the Grassmannians, see   \cite{Kap} and \cite{Kap2}) of varieties that have full strongly exceptional collections of vector bundles but that  cannot have full strongly exceptional collections of line bundles. In \cite{CMZ}, \cite{CMPAMS}, \cite{CM} and \cite{CM1}, two of the authors constructed full strongly exceptional collections of line bundles on smooth toric varieties with a splitting fan and on smooth complete toric varieties with small Picard number. On \cite{CMPAMS} we constructed tilting bundles on projective bundles. The goal of this paper is to construct a full strongly exceptional collection consisting of line bundles
for Zariski locally trivial fibrations. More precisely in the next section, after recalling the notion of a full strongly exceptional collection, we will show the following result.

\begin{theorem}
Let $\phi: X \rightarrow Z$ be a Zariski locally trivial fibration of smooth complex projective varieties with fiber $F$. Assume that $F$ and $Z$ have  a full  strongly exceptional collection of line bundles. Then, there is a full strongly exceptional collection of line bundles on $X$.
\end{theorem}

\vspace{3mm}
Two particular cases of this Theorem were proved by two of the authors in \cite{CMZ}. Indeed, in \cite[Theorem 4.17]{CMZ},  the case where $X$ is the trivial fibration over $Z$ with fiber $F$ (i.e. $X\cong Z\times F$) was proven, and in  \cite[Main Theorem]{CMPAMS},   the case where $F\cong \PP^m$ for some integer $m$ was covered.
\vspace{3mm}

We end the paper observing that for toric fibrations $\phi: X \rightarrow \PP^n$ over $\PP^n$ an explicit   full strongly exceptional collection of line bundles can be provided.

\vspace{3mm}

\noindent {\em Acknowledgment:} The authors thank the referee for his/her useful comments which have improved considerably our work. This work was initiated during a visit of the second author at the University of Barcelona, which she  thanks for its hospitality.


\section{Derived category of Zariski locally trivial fibrations}

The goal of this section is to give a structure theorem for the derived category $D^b(X)$ of a Zariski locally trivial fibration $X$ over $Z$ with fiber $F$. This will be achieved  constructing a full strongly exceptional collection
of line bundles on $X$. We start by recalling the
notions of exceptional sheaves, exceptional collections of
sheaves, strongly exceptional collections of sheaves and full
strongly exceptional collections of sheaves.

\begin{definition}\label{exceptcoll}
Let $X$ be a smooth projective variety.

(i) A  coherent sheaf $E$ on  $X$ is {\em exceptional} if $\Hom
(E,E)=K $ and $\Ext^{i}_{\cO _X}(E,E)=0$ for $i>0$,

(ii) An ordered collection $(E_0,E_1,\ldots ,E_m)$  of coherent
sheaves on $X$ is an {\em exceptional collection} if each sheaf
$E_{i}$ is exceptional and $\Ext^i_{\cO _X}(E_{k},E_{j})=0$ for
$j<k$ and $i \geq 0$.

(iii) An exceptional collection $(E_0,E_1,\ldots ,E_m)$ is a {\em
strongly exceptional collection} if in addition $\Ext^{i}_{\cO
_X}(E_j,E_k)=0$ for $i\ge 1$ and  $j \leq k$.

(iv) An ordered collection $(E_0,E_1,\ldots ,E_m)$ of coherent
sheaves on $X$ is a {\em full (strongly) exceptional collection}
if it is a (strongly) exceptional collection  and $E_0$, $E_1$,
$\ldots $ , $E_m$ generate the bounded derived category $D^b(X)$.
\end{definition}

As an immediate consequence of the fact that for any smooth
projective variety $X$ and for any line bundle $\cL$ on $X$  to twist
by $\cL$ is an equivalence of the derived category $D^b(X)$, we
have

\begin{lemma} \label{effective} Let $X$ be a smooth projective variety and $\cM$ a line bundle on $X$. If $(E_0,E_1,\ldots ,E_m)$ is a  full (strongly) exceptional collection of line bundles  on $X$ then $(E_0\otimes \cM,E_1\otimes \cM,\ldots ,E_m\otimes \cM)$ is also a  full (strongly) exceptional collection of line bundles on $X$.
\end{lemma}

\begin{remark} \label{length}
 As mentioned in the Introduction, the existence of a full
strongly exceptional collection $(E_0,E_1,\ldots,E_m)$ of coherent
sheaves on a smooth projective variety $X$ imposes a rather strong
restriction on $X$, namely that the Grothendieck group
$K_0(X)=K_0({\cO}_X$-mod$)$ is isomorphic to $\ZZ^{m+1}$.
\end{remark}

Let us illustrate the above definitions with  examples:

\begin{example} \label{prihirse}

(1)The collection of line bundles
$(\cO_{\PP^n}, \cO_{\PP^n}(1), \ldots, \cO_{\PP^n}(n))$ is a full strongly exceptional collection of line bundles on
$\PP^n$.

\vskip 2mm (2) Let $X_1$ and $X_2$ be two smooth projective
varieties and let $(F_0^{i},F_1^{i},\ldots ,F_{n_{i}}^{i})$ be   a
full strongly exceptional collection of locally free sheaves on
$X_i$, $i=1,2$.  Then, $$(F_0^{1}\boxtimes
F_0^{2},F_1^{1}\boxtimes F_0^{2},\ldots ,F_{n_1}^{1}\boxtimes
F_0^{2},  \ldots , F_0^{1}\boxtimes F_{n_2}^{2},F_1^{1}\boxtimes
F_{n_2}^{2},\ldots ,F_{n_1}^{1}\boxtimes F_{n_2}^{2})$$ is a full
strongly exceptional collection of locally free sheaves on $X_1
\times X_2$ (see \cite{CMZ}; Proposition 4.16). In particular,
the collection of line bundles
$$( \cO_{\PP^n} \boxtimes \cO_{\PP^m},
\cO_{\PP^n}(1) \boxtimes \cO_{\PP^m},
\ldots, \cO_{\PP^n}(n) \boxtimes \cO_{\PP^m},
\ldots,  \cO_{\PP^n} \boxtimes \cO_{\PP^m}(m), $$ $$ \hspace{20mm} \cO_{\PP^n}(1)
\boxtimes \cO_{\PP^m}(m),\ldots, \cO_{\PP^n}(n) \boxtimes
\cO_{\PP^m}(m)) $$ is a full strongly exceptional collection of
line bundles on $\PP^n \times \PP^m$.

\vskip 2mm (3) Let $X$ be a smooth complete variety which is
the projectivization of a rank $r$ vector bundle $\cE$ over a
smooth complete variety $Z$ which has a full strongly
exceptional collection of locally free sheaves. Then, $X$ also has
a full strongly exceptional collection of locally free sheaves
(See \cite{CMZ}; Proposition 4.9, \cite{CMPAMS}; Main Theorem).

\end{example}

\vspace{4mm}

As we said in the introduction $D^b(X)$ is an important algebraic invariant of a smooth projective variety but very little is known about the structure of $D^b(X)$. In particular, whether  $D^b(X)$ is freely and finitely
generated and, hence, we are lead to consider the following
problem

\begin{problem}\label{prob1}
Characterize smooth projective varieties $X$ which have a full
strongly exceptional collection of coherent sheaves and, even
more, investigate if there is one consisting of line bundles.
\end{problem}

With full generality, this problem seems out of reach and only some particular cases have been addressed. For instance, in \cite{Kaw}, Kawamata proved that the derived category of a
smooth complete toric variety has a full exceptional collection of
objects. In his collection the objects are sheaves rather than
line bundles and the collection is only exceptional and not
strongly exceptional. In the toric context,  there are a lot of
contributions to the above problem. For instance, it turns out
that a full strongly exceptional collection made up of line
bundles exists on  projective spaces (\cite{Be}), products of projective
spaces (\cite{CMZ}; Proposition 4.16), smooth complete toric
varieties with Picard number $\le 2$ (\cite{CMZ}; Corollary 4.13), smooth complete toric varieties with a splitting fan
(\cite{CMZ}; Theorem 4.12) and some smooth complete toric varieties with Picard Number 3 (\cite{MM}). As mentioned in the introduction, not any variety has a full strongly exceptional collection made up of line bundles. For instance,  in \cite{HP}, Hille and Perling
constructed an example of a smooth non Fano toric surface which does
not have a full strongly exceptional collection made up of line
bundles. Recently in \cite{Mateusz} Michalek has also proved that for a big enough $n,$ the blow up of $\PP^n$  at two invariant points does not have  strongly exceptional collection of line bundles. See also \cite{Ef} for  similar results on the Fano context.

\vspace{3mm}

Our goal is to describe the derived category of a Zariski locally trivial fibration. We will contribute to both parts of Problem \ref{prob1} in the sense that we will address  the problem of the existence of  full strongly exceptional collections of coherent sheaves and, in our view more importantly,   the existence of at least one consisting of line bundles.

\begin{remark} Let $\phi: X \rightarrow Z$ be a Zariski locally trivial fibration of smooth complex
projective varieties with fiber $F$. Notice that any line bundle
$\cL$ on $F$  can be lifted to a line bundle on $X$. Indeed, let
$D$ be a Weil divisor associated to $\cL$. We can assume that $D$
is effective and by linearity we extend to any other Weil divisor.
We consider the trivialization $U_X=U_Z\times F$ where $U_X\subset
X$ and $U_Z\subset Z$ are Zariski open subsets and the divisor
$D\times U_Z$ on $U_X$. Its closure in $X$ is a divisor $E$ on $X$
whose restriction to $F$ and to any fiber over $U_Z$ is $D$. By
semicontinuity $\cO _X(E)_{|F}\cong \cO_F(D)\cong \cO_X(E)_{|X_z}$
for any fiber $X_z$ of $\phi$.
\end{remark}

\begin{notation} Let $\phi: X \rightarrow Z$ be a
Zariski locally trivial fibration of smooth complex projective
varieties with fiber $F$. For any line bundle $\tilde{\cL}$  on
the fiber $F$ we will denote by $\cL$ its lifting to $X$.
\end{notation}

We are now ready to state the main result of this paper.

\vspace{3mm}

\begin{theorem} \label{mainthm} Let $\phi: X \rightarrow Z$ be a Zariski locally trivial fibration of smooth complex projective varieties with fiber $F$. Let  $(\tilde{{\cL_1}}, \ldots , \tilde{{\cL_u})}$ be  a full  strongly exceptional
collection of line bundles on $F$ and let $(\cE_1, \ldots , \cE_v)$ be a  full strongly exceptional collection of line bundles on $Z.$
Then, there is a line bundle $\cO_Z(D)$ on $Z$ such that
the ordered sequence
$$\{\{ \phi^*(\cE_j \otimes \cO_Z(D))\otimes \cL_1\}_{ 1\leq j\leq v}, \{\phi^*(\cE_j \otimes \cO_Z(2D))\otimes \cL_2\}_{ 1\leq j\leq v}, \ldots, \{ \phi^*(\cE_j \otimes \cO_Z(uD))\otimes \cL_u\}_{ 1\leq j\leq v}\}$$
is a full strongly exceptional sequence of line bundles on $X$.
\end{theorem}
\begin{proof}

It is a general fact that given a line bundle $\cL$ on a Zariski  locally trivial fibration on a smooth variety with fiber $F$ then $\cL|_{F}=\cL|_{X_z}$ for any fiber $X_z.$ Therefore, $({\cL_1}, \ldots , {\cL_u})$ is a collection of line bundles on $X$  that reduces to a full  strongly exceptional collection of line bundles on all the fibers.

\vspace{3mm}

We have to prove the existence of a line bundle $\cO_Z(D)$ on $Z$ such that the above collection is full and satisfies the following cohomological conditions:

(a) $\Ext^k( \phi^*(\cE_i \otimes \cO_Z(jD))\otimes \cL_j, \phi^*(\cE_p \otimes \cO_Z(qD)) \otimes \cL_q) =0$ for $k > 0$ and $q \geq j$, and

(b) $\Ext^k(\phi^*(\cE_i \otimes \cO_Z(jD)) \otimes \cL_j , \phi^*(\cE_p \otimes \cO_Z(qD)) \otimes \cL_q)=0$ for $k \geq 0$ and $q < j$,

\noindent which are equivalent to
\[ (a') \quad H^k(X, \phi^*(\cE_p \otimes \cE_i^{\vee} \otimes \cO_Z((q-j)D))\otimes \cL_q \otimes \cL_j^{\vee}) =0, \quad \mbox{for } k >0, \quad q \geq j \]
\[ (b') \quad H^k(X, \phi^*(\cE_p \otimes \cE_i^{\vee} \otimes \cO_Z((q-j)D))\otimes \cL_q \otimes \cL_j^{\vee}) =0, \quad \mbox{for } k \geq 0, \quad q < j. \]
First of all recall that the Leray spectral sequence for $\phi$ gives us:
\begin{equation} \label{leray}  \begin{array}{r} H^r(Z, R^s \phi_*(\phi^*(\cE_p \otimes \cE_i^{\vee} \otimes \cO_Z((q-j)D))\otimes \cL_q \otimes \cL_j^{\vee})) \hspace{30mm}  \\  \Rightarrow
H^{r+s}(X, \phi^*(\cE_p \otimes \cE_i^{\vee} \otimes \cO_Z((q-j)D))\otimes \cL_q \otimes \cL_j^{\vee})\end{array} \end{equation}
and by the projection formula
\begin{equation} \label{projection} R^s \phi_*(\phi^*(\cE_p \otimes \cE_i^{\vee} \otimes \cO_Z((q-j)D))\otimes \cL_q \otimes \cL_j^{\vee})=
 \cE_p \otimes \cE_i^{\vee} \otimes \cO_Z((q-j)D)\otimes R^s \phi_*(\cL_q \otimes \cL_j^{\vee}). \end{equation}

 \vspace{3mm}

 Assume that $q<j$. Since $({\cL_1}, \ldots , {\cL_u})$ is a collection of line bundles on $X$  that reduces to a full  strongly exceptional collection of line bundles on all fibers $X_z$, $z \in Z$, for any $s \geq 0$
\[ H^s(X_z, \cL_q \otimes \cL_j^{\vee} \otimes \cO_{X_z})=0\]
and thus $R^s \phi_* (\cL_q \otimes \cL_j^{\vee})=0$ for any $s \geq 0$. Therefore, by (\ref{projection}) for any $s \geq 0$ we have
 \[ R^s \phi_*(\phi^*(\cE_p \otimes \cE_i^{\vee} \otimes \cO_Z((q-j)D))\otimes \cL_q \otimes \cL_j^{\vee})=0 \]
 which by the Leray spectral sequence implies that
\[ H^k(X, \phi^*(\cE_p \otimes \cE_i^{\vee} \otimes \cO_Z((q-j)D))\otimes \cL_q \otimes \cL_j^{\vee}) =0, \quad \mbox{for } k \geq 0, \quad q < j \]
which proves $(b')$.

\vspace{3mm}

Assume that $q=j$.  Arguing as in the above case, since  $({\cL_1}, \ldots , {\cL_u})$ is a collection of line bundles on $X$  that reduces to a full  strongly exceptional collection of line bundles on all fibers $X_z$, we get that

 \[ R^s \phi_* (\cL_q \otimes \cL_q^{\vee})=  \left \{ \begin{array}{ll} 0 & \mbox{for} \quad s >0 \\
 \cO_Z  & \mbox{for} \quad s=0. \end{array} \right.   \]
 Hence, by the projection formula (\ref{projection}) and using the fact that $(\cE_1, \ldots , \cE_v)$ is a  full strongly exceptional collection of line bundles on  $Z$ we get that
\[ H^r(Z, R^s \phi_*(\phi^*(\cE_p \otimes \cE_i^{\vee})\otimes \cL_q \otimes \cL_q^{\vee}))=0 \]
unless $r=s=0$. Therefore, by the Leray spectral sequence (\ref{leray})
 \[ H^k(X, \phi^*(\cE_p \otimes \cE_i^{\vee})\otimes \cL_q \otimes \cL_q^{\vee}) =0, \quad \mbox{for } k >0 \]
 which proves $(a')$ in case $j=q$.

\vspace{3mm}

Finally assume that $q>j$. Since the line bundles $\cE_p$, $1 \leq p \leq v$ and $\cL_q$, $1\leq q \leq u$ move in a finite set, there exists an ample line bundle $\cO_Z(D)$ on $Z$ such that for any $s \geq 0$, $r>0$ and $j-q>0$,
\[ H^r(Z, \cE_p \otimes \cE_i^{\vee} \otimes \cO_Z((q-j)D)\otimes R^s\phi_*(\cL_q \otimes \cL_j^{\vee})) =0\]
and
\[ H^r(Z, \phi_*(\phi^*(\cE_p \otimes \cE_i^{\vee}) \otimes \cL_q \otimes \cL_j^{\vee})\otimes \cO_Z((q-j)D)) =0.\]
Hence, it follows from the projection formula (\ref{projection}) and the Leray spectral sequence (\ref{leray}) that
 \[ H^k(X, \phi^*(\cE_p \otimes \cE_i^{\vee})\otimes \cL_q \otimes \cL_q^{\vee}) =0, \quad \mbox{for } k >0 \]
 which finishes the proof of $(a').$

\vspace{3mm}

Putting altogether we get that indeed there exists an ample line bundle $\cO_Z(D)$ on $Z$ such that the given ordered  collection of line bundles is strongly exceptional.

\vspace{3mm}
Let us prove that it is full. To this end, observe that since by assumption $(\cE_1, \ldots,\cE_v )$ is a full strongly exceptional collection on $Z$, by Lemma \ref{effective},   for any $j>0$ we have
\[ D^b(Z)= \langle \cE_1 \otimes \cO_Z(jD), \ldots,\cE_v \otimes \cO_Z(jD) \rangle.\]
Therefore, it suffices  to prove that
\[D^b(X)=\langle \phi^*D^b(Z) \otimes \cL_1, \ldots, \phi^*D^b(Z) \otimes \cL_u \rangle \]
where $\phi^*D^b(Z) \otimes \cL_k$ denotes the full triangulated subcategory in $D^b(X)$ generated by the objects
$\{ \phi^*(\cE_j)\otimes \cL_k \}_{ 1\leq j\leq v}$.

To this end, it is enough to check that  $ \langle \phi^*D^b(Z) \otimes \cL_1, \ldots, \phi^*D^b(Z) \otimes \cL_u \rangle$ contains all the objects $\cO_x$, $x \in X$ since the set $\{\cO_x| x \in X \}$ is a spanning class for $D^b(X)$. Since every point $x \in X$ belongs to some fiber $X_z:= \phi^{-1}(z)$, $z \in Z$, and we have already seen that
$({\cL_1}, \ldots , {\cL_u})$ reduces to a collection $(\tilde{\cL_1}, \ldots , \tilde{\cL_u})$ that generates $D^b(X_z)$, we get
\[ \cO_x \in \langle \tilde{\cL_1}, \ldots , \tilde{\cL_u} \rangle  .\]
So, since $\phi^*(\cO_z)=\cO_{X_z}$, the sheaf $\cO_x$ belongs to
$$\langle \phi^*D^b(Z) \otimes \cL_1, \ldots, \phi^*D^b(Z) \otimes \cL_u \rangle.$$
Putting altogether we get that
$$\{\{ \phi^*(\cE_j \otimes \cO_Z(D))\otimes \cL_1\}_{ 1\leq j\leq v}, \{\phi^*(\cE_j \otimes \cO_Z(2D))\otimes \cL_2\}_{ 1\leq j\leq v}, \ldots, \{ \phi^*(\cE_j \otimes \cO_Z(uD))\otimes \cL_u\}_{ 1\leq j\leq v}\}$$
generates $D^b(X)$.
  \end{proof}

\vspace{3mm}

   \begin{remark}
   Moreover we observe that for toric fibrations over $\PP^n$
  a stronger result is true. Indeed, consider
 $\phi: X \rightarrow \PP^n$   a toric $F$-fibration over $\PP^n$. Let  $(\tilde{\cL_1}, \ldots , \tilde{\cL_u})$ be a full  strongly exceptional
collection of line bundles on $F$ and take on $\PP^n$ the usual  full strongly exceptional collection of line bundles  $(\cO_{\PP^n}, \ldots , \cO_{\PP^n}(n))$. Then in this particular case, following the same ideas as in the proof of Theorem \ref{mainthm} and strongly using the fact that we have an explicit control of the cohomology of line bundles on this toric fibration, we get that the ordered sequence
$$\{\{ \phi^*(\cO_{\PP^n})\otimes \cL_1\}_{ 1\leq j\leq v}, \{\phi^*(\cO_{\PP^n}(1))\otimes \cL_2\}_{ 1\leq j\leq v}, \ldots, \{ \phi^*(\cO_{\PP^n}(n) )\otimes \cL_u\}_{ 1\leq j\leq v}\}$$
is a full strongly exceptional sequence of line bundles on $X$.
\end{remark}



\begin{thebibliography}{999}

\bibitem{Ba} D. Baer, {\em  Tilting sheaves}, Manusc. Math. {\bf 60} (1988), 323a-347.







\bibitem{Be} A.A. Beilinson, {\em Coherent sheaves on $\PP^n$ and Problems of Linear Algebra},
Funkt. Anal. Appl. {\bf 12} (1979), 214-216.


 \bibitem{Bo} A.I. Bondal, {\em Representation of associative algebras
  and coherent sheaves}, Math. USSR Izvestiya {\bf 34} (1990),
  23-42.



\bibitem{BO}  A. Bondal and D. Orlov, {\em  Semiorthogonal decomposition for algebraic varieties}, preprint:math.AG/9506012.

\bibitem{CMZ} L. Costa and R.M. Mir\'o-Roig, {\em Tilting sheaves
on toric varieties}, Math. Z. {\bf 248} (2004), 849-865.

\bibitem{CMPAMS} L. Costa and R.M. Mir\'o-Roig, {\em Derived categories of projective
bundles},  Proc. Amer. Math. Soc.  {\bf 133}  (2005), 2533-2537.

\bibitem{CM} L. Costa and R.M. Mir\'o-Roig, {\em Derived category of Toric varieties
with small Picard number}, Preprint 2008.

\bibitem{CM1} L. Costa and R.M. Mir\'o-Roig, {\em Frobenius splitting and Derived category of Toric varieties}, To appear Illinois Journal of Mathematics.

\bibitem{bh} L. Borisov, Z. Hua, {\em On the conjecure of King for
smmoth toric Deligne-Mumford stacks}, math.AG/0801.2812v3.


\bibitem{Ef} A. Efimov, {\em Maximal lengths of exceptional collections of line bundles }, arXiv:1010.3755 , 2010.

\bibitem{Ka} A.N. Rudakov et al. {\em Helices and vector bundles: Seminaire Rudakov}
Lecture Note Series, {\bf 148} (1990) Cambridge University Press.


\bibitem{HP} L. Hille, M. Perling, {\em A counterexample to King's
conjecture}, Compos. Math. {\bf 142} (2006),  1507-1521.

\bibitem{Kap} M. M. Kapranov, {\em On the derived category of coherent sheaves on Grassmann manifolds}, Math.
USSR Izvestiya, {\bf 24} (1985), 183-192.
\bibitem{Kap2} M. M. Kapranov, {\em On the derived category of coherent sheaves on some homogeneous spaces}, Invent.
Math., 92 (1988), 479-508.

\bibitem{Kaw} Y. Kawamata, {\em Derived categories of toric varieties}, Michigan
Math. J. {\bf 54} (2006), 517-535.

\bibitem{Ko} M. Kontsevich, {\em Homological algebra of mirror symmetry},  Proceedings of the International Congress of Mathematicians, Vol. 1, 2 (1994),  120--139, Birkh{\"a}user, Basel, 1995.

\bibitem{MM} M. Lason, M. Michalek, {\em On the full, strongly exceptional collections on toric varieties with Picard number three}, arXiv:1003.2047, 2010.

\bibitem{Mateusz} M. Michalek, {\em Family of counterexamples to King's conjecture}, Preprint 2010.


\bibitem{Or} D.O. Orlov, {\em Projective bundles, monoidal transformations, and derived
categories of coherent sheaves}, Math. USSR Izv. {\bf 38} (1993),
133-141.


\end{thebibliography}
\end{document}